\documentclass[11pt]{amsart}
\usepackage[utf8]{inputenc}
\usepackage{fontenc}
\usepackage{amsfonts}
\usepackage{amssymb}
\usepackage{amsmath}
\usepackage{amsthm}\usepackage{mathtools}
\usepackage{enumerate}
\usepackage{hyperref}\usepackage{enumitem}
\usepackage{mathrsfs}
\usepackage{tikz}
\usepackage{marginnote}
\usepackage{xcolor}
\usepackage{soul}

\usepackage[all,cmtip]{xy}

\newcommand{\N}{\mathbb{N}}

\newtheorem*{rigprob*}{Rigidity Problem for Uniform Roe Algebras}
\newtheorem*{rigprobcorona*}{Rigidity Problem for Uniform Roe Coronas}

\newtheorem{theorem}{Theorem}[section]
\newtheorem*{theorem*}{Theorem}
\newtheorem{proposition}[theorem]{Proposition}

\newtheorem*{proposition*}{Proposition}
\newtheorem{lemma}[theorem]{Lemma}
\newtheorem*{lemma*}{Lemma}
\newtheorem{corollary}[theorem]{Corollary}
\newtheorem*{corollary*}{Corollar}

\newtheorem*{fact*}{Fact}
\theoremstyle{definition}
\newtheorem{definition}[theorem]{Definition}
\newtheorem*{definition*}{Definition}

\newtheorem*{claim*}{Claim}

\newtheorem*{conjecture*}{Conjecture}

\theoremstyle{remark}
\newtheorem{example}{Example}
\newtheorem*{example*}{Example}
\newtheorem{remark}[theorem]{Remark}
\newtheorem*{remark*}{Remark}

\newtheorem*{note*}{Note}
\newtheorem*{question*}{Question}

\DeclareMathOperator{\Span}{span}

\DeclareMathOperator{\supp}{supp}

\newcounter{my_enumerate_counter}
\newcommand{\pushcounter}{\setcounter{my_enumerate_counter}{\value{enumi}}}
\newcommand{\popcounter}{\setcounter{enumi}{\value{my_enumerate_counter}}}

\usepackage{enumitem}

\begin{document}

\title[Shrinking Schauder Frames and their Associated Spaces]{Shrinking Schauder Frames and their Associated Bases}%

\author{Kevin Beanland}
\address[K. Beanland]{Department of Mathematics, Washington \& Lee University, 204 W. Washington St. Lexington, VA, 24450}
\email{beanlandk@wlu.edu}
\urladdr{kbeanland.academic.wlu.edu}

\author{Daniel Freeman}
\address[D. Freeman] {Department of Mathematics and Statistics, St Louis University, St Louis, MO 		}
\email{daniel.freeman@slu.edu}
\urladdr{mathstat.slu.edu/~freeman/}

\thanks{The first author was supported by a Lenfest Summer Research Grant through Washington \& Lee University. The second author was supported by grant 706481 from the Simons Foundation.}
\date{\today}%
\maketitle

\begin{abstract}
    For a Banach space $X$ with a shrinking Schauder frame $(x_i,f_i)$ we provide an explicit method for constructing a shrinking associated basis.  In the case that the minimal associated basis is not shrinking, we prove that every shrinking associated basis of $(x_i,f_i)$ dominates an uncountable family of incomparable shrinking associated bases of $(x_i,f_i)$. By adapting a construction of Pe{\l}czy{\'n}ski, we characterize spaces with shrinking Schauder frames as space having the $w^*$-bounded approximation property. 
    
\end{abstract}

\setcounter{tocdepth}{1}

\section{Introduction}

A frame for a separable infinite dimensional Hilbert space $H$ is a sequence of vectors $(x_j)_{j=1}^\infty$ in $H$ such that there exists constants $0<A\leq B$ so that $A\|x\|^2\leq\sum |\langle x, x_j\rangle|^2\leq B\|x\|^2$ for all $x\in H$.  If $(x_j)_{j=1}^\infty$ is a frame of $H$ then there exists a possibly different frame $(f_j)_{j=1}^\infty$ of $H$ called a {\em dual frame} such that 
\begin{equation}\label{E:frame}
    x=\sum_{j=1}^\infty \langle f_j,x\rangle x_j\hspace{1cm}\textrm{ for all }x\in H.
\end{equation}
That is, frames can be used  like a basis to give a linear reconstruction formula for vectors in $H$.  The difference between frames and bases is that a frame allows for redundancy in that the coefficients given for reconstruction in \eqref{E:frame} are not required to be unique.  

Frames have been generalized to Banach spaces in various ways such as atomic decompositions \cite{FG-atomic,G-atomic}, framings \cite{CHL}, and Schauder frames \cite{CCS-JMAA,CDOS}.  We will focus on Schauder frames in this paper which are a direct generalization of the reconstruction formula given in \eqref{E:frame}. 
Let $X$ be a separable infinite dimensional Banach space. A sequence of pairs $(x_j,f_j)_{j=1}^\infty$ in $X \times X^*$ is called a {\em Schauder frame} for $X$ if
\begin{equation}
        x=\sum_{j=1}^\infty f_j(x) x_j\hspace{1cm}\textrm{ for all }x\in X.
\end{equation}
 We make the convention that $x_j\not=0$ for $j \in \mathbb{N}$.   Though Schauder frames were not explicitly defined until 2008, the first appearance of a Schauder frame, without the name Schauder frame, is in A. Pe{\l}czy{\'n}ski's proof \cite{Pel-BAP} that every space with the bounded approximation property (BAP) is isomorphic to complemented subspace of a space with a basis\footnote{In the same year (1971) Johnson, Rosenthal, and Zippin \cite{JRZ-Israel} proved the same result with a completely different method.}. In 1987, S. Szarek showed that spaces with the BAP but without bases exist \cite{Szarek-Acta}. 
Indeed, Pe{\l}czy{\'n}ski showed that a separable Banach space $X$ has the BAP if and only if it has a Schauder frame (in the above sense) and that, furthermore, there is a space $Z$ with a basis so that the identity on $X$ factors through the identity on $Z$ in a natural way. Formally, if $(x_i,f_i)_{i=1}^\infty$ is a Schauder frame for $X$ and $Z$ is a Banach space with a Schauder basis $(z_i)_{i=1}^\infty$ then $Z$ is an {\em associated space} for $(x_i,f_i)_{i=1}^\infty$ (and $(z_i)_{i=1}^\infty$ is an {\em associated basis} for the frame $(x_i,f_i)_{i=1}^\infty$) if the maps $T:X \to Z$ (analysis operator) and $S:Z\to X$ (synthesis operator) defined by $T(\sum_{i=1}^\infty f_i(x)x_i)=\sum_{i=1}^\infty f_i(x)z_i$ and $S(\sum_{i=1}^\infty a_i z_i)= \sum_{i=1}^\infty a_i x_i$ are bounded.  

In the case that both $(x_i)_{i=1}^\infty$ and $(f_i)_{i=1}^\infty$ are frames for a Hilbert space $H$, then the associated space $Z$ can be chosen to be $\ell_2$ and the associated basis $(z_i)_{i=1}^\infty$  can be chosen to be the unit vector basis for $\ell_2$.  This is of fundamental importance in frame theory as well as in applications such as signal processing.  Indeed, given some vector $x\in H$, the analysis operator maps $x$ to $T(x)=(\langle f_i,x\rangle)_{i=1}^\infty\in\ell_2$.  One can then apply filters to the sequence of frame coefficients 
$(\langle f_i,x\rangle)_{i=1}^\infty$ to obtain a sequence $(b_i)_{i=1}^\infty\in\ell_2$.  Applying the synthesis operator then gives a vector $S((b_i)_{i=1}^\infty)=\sum_{i=1}^\infty b_i x_i$ which is an approximation of $x$ but is improved in some way such as being compressed or having noise or artifacts removed.

If one wishes to use similar techniques for a Schauder frame $(x_i,f_i)_{i=1}^\infty$ then it is advantageous to construct the associated basis $(z_i)_{i=1}^\infty$ to be as nice as possible.  That is, if $(x_i,f_i)_{i=1}^\infty$ has some desirable property such as being unconditional, shrinking, or boundedly complete then one would like $(z_i)_{i=1}^\infty$ to share the property as well.  If $(x_i,f_i)_{i=1}^\infty$ is unconditional, then it is straightforward to construct an unconditional associated basis.  In \cite{BFL-Fund}, the authors of the current paper and R. Liu prove that if $(x_i,f_i)_{i=1}^\infty$ is shrinking  then it has a shrinking associated basis.  However, the construction in \cite{BFL-Fund} is relatively difficult and involves the method of  bounds on branches of weakly null trees developed by E. Odell and Th. Schlumprecht \cite{OS-trees} \cite{FrOScZ-Fund}.  One of the main goals of this paper is to give  a more direct and much simpler construction of a shrinking associated basis, which we state in the following theorem.

\begin{theorem}\label{T:intro}
Let $(x_i,f_i)_{i=1}^\infty$ be a shrinking Schauder frame for a Banach space $X$.  For each $m\leqslant n$, we denote $P_{[m,n]}:X\rightarrow X$ to be the operator $P_{[m,n]}(x)=\sum_{i=m}^n f_i(x) x_i$.  Then there exists an increasing sequence of natural numbers $(N_k)_{k=1}^\infty$ such that
\begin{equation}\label{E:introT}
\sup_{\substack{m_0<n_0 \leqslant k \\ N_k \leqslant m \leqslant n}}\|P_{[m_0,n_0]} P_{[m,n]} x\|\leqslant 2^{-k}\|x\|\hspace{1cm}\textrm{ for all }x\in X.
\end{equation}
Furthermore,  if $(N_k)_{k=1}^\infty$ satisfies \eqref{E:introT} then $Z_{(N_k)}$ is an associated space of $(x_i,f_i)_{i=1}^\infty$ and  $(z_i)_{i=1}^\infty$ is a shrinking associated basis  where for $\sum a_i z_i\in Z_{(N_k)}$, the norm is given by
    \begin{equation}\label{E:intro}
            \Big\|\sum a_i z_i\Big\|_{(N_k)}=\sup_{m\leqslant n}\Big\|\sum_{m\leqslant i\leqslant n} a_i x_i\Big\| \vee 
    \sup_{\substack{m_0\leqslant n_0\leq k\\ N_k \leqslant m\leqslant n}} 2^k \Big\|P_{[m_0,n_0]}  \sum_{m\leqslant i\leqslant n} a_i x_i\Big\|.
        \end{equation}

\end{theorem}

Given a Schuader frame $(x_i,f_i)_{i=1}^\infty$ with $x_i\neq 0$ for all $i\in\N$, the most natural associated space  is now referred to as the {\em minimal associated space} \cite{CDOS}\cite{Liu-Frames} and is defined as follows. Denote by $(z_i)_{i=1}^\infty$  the unit vector basis for $c_{00}$ and for $(a_i)\in c_{00}$ consider the norm
\begin{equation}\label{E:min}
\Big\|\sum a_i z_i\Big\|_{\min}=\sup_{m<n}\Big\|\sum_{m\leqslant i \leqslant n} a_i x_i\Big\|. 
\end{equation}
The minimal associated space $Z_{min}$ is defined to be the completion of $c_{00}$ under the above norm and the basis $(z_i)_{i=1}^\infty$ is called the {\em minimal associated basis}.
A Schauder frame $(x_i,f_i)_{i=1}^\infty$ may have many non-equivalent associated bases.  However, 
the basis $(z_i)_{i=1}^\infty$ defined in \eqref{E:min} is  minimal in the sense that if $(y_i)_{i=1}^\infty$ is any associated basis for $(x_i,f_i)_{i=1}^\infty$ then $(y_i)_{i=1}^\infty$ dominates $(z_i)_{i=1}^\infty$. That is, there exists a constant $K>0$ so that $\|\sum_{i} a_i z_i\|\leq K\|\sum_{i} a_i y_i\|$ for all $(a_i)\in c_{00}$ \cite{Liu-Frames}.  

We now consider the problem of determining if a shrinking Schauder frame has a minimal shrinking associated basis.  That is, if $(x_i,f_i)_{i=1}^\infty$ is a shrinking Schauder frame, then when does there exist a shrinking associated basis $(w_i)_{i=1}^\infty$ of $(x_i,f_i)_{i=1}^\infty$ such that if $(y_i)_{i=1}^\infty$ is any shrinking associated basis of $(x_i,f_i)_{i=1}^\infty$ then $(y_i)_{i=1}^\infty$ dominates $(w_i)_{i=1}^\infty$?  In Section \ref{S:compare} we prove that a Schauder frame has a minimal shrinking associated basis if and only if the minimal associated basis defined in \eqref{E:min} is shrinking. Our construction of a shrinking associated basis in Theorem \ref{T:intro} is defined solely in terms of the shrinking Schauder frame $(x_i,f_i)_{i=1}^\infty$ and some sequence of natural numbers $(N_k)_{k=1}^\infty\in[\N]^\omega$.  In Section \ref{S:compare} we prove that if  $(x_i,f_i)_{i=1}^\infty$ is a shrinking Schauder frame  and $(y_i)_{i=1}^\infty$ is any shrinking associated basis then there exists $(N_k)_{k=1}^\infty\in[\N]^\omega$ such that the resulting shrinking associated basis $(z_i)_{i=1}^\infty$ from our construction is dominated by $(y_i)_{i=1}^\infty$.  In other words,  Theorem \ref{T:intro} produces a set of shrinking associated bases such that every shrinking associated basis of $(x_i,f_i)_{i=1}^\infty$  dominates some basis in that set.  Furthermore, we prove that if the minimal associated basis is not shrinking then for every shrinking associated basis $(y_i)_{i=1}^\infty$  there exists uncountably many mutually incomparable shrinking associated bases which are all dominated by $(y_i)_{i=1}^\infty$.  Hence, except for the trivial case where the minimal associated basis is shrinking, we have that the collection of shrinking associated bases will have a very rich lattice structure under the domination partial order.

In the final section we make some observations about how this work relates to now classical results about the BAP and give an alternative proof of the theorem of Johnson, Rosenthal, and Zippin that for a Banach space $X$ with separable dual, $X^*$ has the BAP if and only if $X$ is isomorphic to a complemented subspace of a Banach space with a shrinking basis.

\section{Shrinking Schauder bases and shrinking Schauder  frames}

A sequence of vectors $(x_i)_{i=1}^\infty$ in a separable Banach space $X$ is called a {\em Schauder basis } if for all $x\in X$ there exists a unique sequence of scalars $(a_i)_{i=1}^\infty$ such that $x=\sum_{i=1}^\infty a_i x_i$.  If $X$ is a Banach space with dual $X^*$ then 
a Schauder basis $(x_i)_{i=1}^\infty$ is called {\em shrinking} if  the biorthogonal functionals $(x_i^*)_{i=1}^\infty$ form a Schauder basis for $X^*$. In particular, a Banach space with a shrinking basis necessarily has a separable dual with a basis. There are, however, Banach spaces with bases whose duals are separable but fail the approximation property \cite[Theorem 1.e.7.(b)]{LTz-book}. Naturally, a Schauder frame $(x_i,f_i)_{i=1}^\infty$ for $X$ is shrinking if and only if $(f_i,x_i)_{i=1}^\infty$ is a Schauder frame for $X^*$. 
Using the terminology atomic decomposition instead of Schauder frame, Carando and Lassalle \cite{CL-Studia} give the following useful characterization of shrinking Schauder frames which is analogous to James' well-known characterization for Schauder bases. 

\begin{theorem}[\cite{CL-Studia}, Theorem 1.4]
Let $(x_i,f_i)_{i=1}^\infty$ be a Schauder frame for $X$. For each interval $I\subseteq\N$, let $P_I:X\rightarrow X$ be the operator $P_I(x)=\sum_{i\in I} f_i(x) x_i$. Then $(f_i,x_i)_{i=1}^\infty$ is a Schauder frame for $X^*$ if and only if for each $f \in X^*$ we have that
\begin{equation}
    \lim_{n \to \infty}\|f\circ P_{[n,\infty)}\|=0.
    \label{reconstruct}
\end{equation}
\end{theorem}

We sketch a short proof for completeness.

\begin{proof}
Consider the reverse direction and assume that \eqref{reconstruct} holds. Let $f\in X^*$.  It suffices to show that  $\sum f(x_i)f_i$ is a Cauchy sequence. This follows readily from (\ref{reconstruct}) as
\begin{align*}
            \limsup_{m,n\to \infty} \Big\|\sum_{i=m}^n f(x_i)f_i\Big\| & = \limsup_{m,n\to \infty} \sup_{x \in S_X} \sum_{i=m}^n f(x_i)f_i(x)\\
                & = \limsup_{m,n\to \infty} \sup_{x \in S_X} f\Big(\sum_{i=m}^n f_i(x)x_i\Big) \\
                & = \limsup_{m,n \to \infty} \|f\circ P_{[m,n]}\|=0.
\end{align*}

Therefore $(f_i,x_i)_{i=1}^\infty$ is a Schauder frame for $X^*$. A similar proof shows that the converse holds. 
\end{proof}

Let $X$ be a Banach space with a Schauder frame $(x_i,f_i)_{i=1}^\infty$ and let $Z$ be a Banach space with a Schauder basis $(z_i)_{i=1}^\infty$.  Recall that $Z$ is said to be an {\em associated space} of the Schauder frame $(x_i,f_i)_{i=1}^\infty$ and $(z_i)_{i=1}^\infty$ is said to an {\em associated basis} if the maps $T:X\rightarrow Z$ and $S:Z\rightarrow X$ are bounded where $T(x)=\sum_{i=1}^\infty f_i(x) x_i$ for all $x\in X$ and $S(\sum_{i=1}^\infty a_i z_i)=\sum_{i=1}^\infty a_i x_i$ for all $\sum_{i=1}^\infty a_i z_i\in Z$.  It follows immediately that if a Schauder frame  has a shrinking associated basis then the Schauder frame must be shrinking.  In \cite{BFL-Fund}, the authors of the current paper and R. Liu prove a more general and technical theorem which implies that every shrinking Schauder frame has a shrinking associated basis.   Unfortunately, the argument does not provide an explicit construction of the associated basis and the proofs are relatively difficult. In Section \ref{S:main} we  give an explicit method which will give a shrinking associated basis for any shrinking Schauder frame.  Before proceeding we show that  the minimal associated basis for a shrinking Schauder frame need not be shrinking.

\begin{example}\label{example}
Let $(e_i)$ be the unit vector basis for $\ell_2$. Let $x_1=e_1$ and $f_1=e_1^*$.  For all $i\in\N$ we let $x_{2i+1}=e_1$,  $x_{2i}=e_{i+1}$, $f_{2i+1}=0$, and $f_{2i}=e^*_{i+1}$. Then $(x_i,f_i)_{i=1}^\infty$ is a shrinking Schauder frame for $\ell_2$ but the minimal associated space for $(x_i,f_i)_{i=1}^\infty$ has the norm
\begin{equation}
    \big\|\sum a_i z_i\big\|_{\min}^2= \sup\bigg\{\big|\sum_{\substack{i\in I\\ i\text{ odd }}} a_{i}\big|^2+ \sum_{\substack{i\in I\\ i\text{ even }}} |a_i|^2 : I \subset \mathbb{N}, \text{ is an interval}\bigg\}.
\end{equation}
The basis $(z_i)$ is not shrinking since the sequence of odd vectors $(z_{2i-1})_{i\in \mathbb{N}}$ is equivalent to the summing basis of $c_0$.
\end{example}

The above example, although simple, is rather instructive in that it reveals that redundancy in a frame can make the minimal associated basis not shrinking.  

A sequence $(y_i)$ in a Banach space is called $\alpha$-$\ell_1^+$ for some $\alpha>0$ if $\|\sum a_i y_i\|\geqslant \alpha\sum a_i$ whenever $(a_i)$ is a summable sequence of non-negative scalars.  We conclude this section by recalling the following useful and well known characterization of shrinking bases.  

\begin{lemma}\label{L:w}
Let $X$ be a Banach spaces with a Schauder basis $(x_i)$. The following are equivalent:
\begin{enumerate}
    \item $(x_i)$ is not shrinking.
    \item There is a normalized block sequence $(y_i)$ of $(x_i)$ that is not weakly null. 
    \item There is a normalized block  sequence $(y_i)$ of $(x_i)$ that is  $\alpha$-$\ell_1^{+}$ for some $\alpha>0$.
\end{enumerate}
\end{lemma}

\section{Constructing a shrinking associated basis}\label{S:main}

We start by setting some notation.  Let $(x_i,f_i)_{i=1}^\infty$ be a shrinking Schauder frame for a Banach space $X$. Let $Z$ be an associated space and $(z_i)_{i=1}^\infty$ be an associated basis.  The {\em analysis operator} is the map $T:X\rightarrow Z$ given by $T(x)=\sum_{i=1}^\infty f_i(x)z_i$ for all $x\in X$.  The {\em synthesis operator} is the map $S:Z\rightarrow X$ given by $S(\sum_{i=1}^\infty a_i z_i)=\sum_{i=1}^\infty a_i x_i$ for all $\sum_{i=1}^\infty a_i z_i\in Z$.  For $m<n$, we use the following notation when we wish to use partial sums.
\begin{enumerate}
    \item $S_{[m,n]}(\sum_i a_i z_i)= \sum_{i=m}^n a_i x_i$ for all $\sum a_iz_i\in Z$, 
    \item $R_{[m,n]}(\sum_i a_i z_i)= \sum_{i=m}^n a_i z_i$ for all $\sum a_iz_i\in Z$, 
\item $P_{[m,n]}(x)= \sum_{i=m}^n f_i(x) x_i$.
\end{enumerate}
It follows from the uniform boundedness principle that $\sup_{m\leq n}\|S_{[m,n]}\|$, $\sup_{m\leq n}\|R_{[m,n]}\|$, and $\sup_{m\leq n}\|P_{[m,n]}\|$ are all finite.  The value $\sup_{m\leq n}\|R_{[m,n]}\|$ is the basis constant of $(z_i)_{i=1}^\infty$ and  the value $\sup_{m\leq n}\|P_{[m,n]}\|$ is called the frame constant of $(x_i,f_i)_{i=1}^\infty$.

The following proposition is contained in \cite{BFL-Fund} and we include the short proof for completeness. 

\begin{proposition}\label{P:N_k}
Let $(x_i,f_i)_{i=1}^\infty$ be a shrinking Schauder frame for a Banach space $X$. Then there is an increasing sequence $(N_k)_{k=1}^\infty$ of natural numbers so that
\begin{equation}
\sup_{\substack{m_0<n_0 \leqslant k \\ N_k <m<n}}\|P_{[m_0,n_0]} P_{[m,n]} x\|\leqslant 2^{-k}\|x\|\hspace{1cm}\textrm{ for all }x\in X.
\label{killstuff}
\end{equation}
\end{proposition}

\begin{proof}
Let $k \in \mathbb{N}$ and $\varepsilon>0$. It suffices to show that there is an $N_k >k$ satisfying \begin{equation*}
\sup_{\substack{m_0\leqslant n_0 \leqslant k \\ N_k \leqslant m\leqslant n}}\|P_{[m_0,n_0]} P_{[m,n]} x\|< \varepsilon \hspace{1cm}\textrm{ for all $x\in X$ with $\|x\|=1$.}
\end{equation*}
As $(x_i,f_i)_{i=1}^\infty$ is shrinking we have that $(f_i,x_i)_{i=1}^\infty$ is a Schauder frame for $X^*$.  Thus, for sufficiently large $N_k$ we have that
$$\sup_{N_k \leqslant  m\leqslant n}\|\sum_{i=m}^n f_i(x_j)f_i\|<\frac{\varepsilon}{k\|x_j\|}\hspace{1cm}\textrm{for all  $1\leqslant j \leqslant k$.}$$
This $N_k$ suffices as for fixed $m_0\leqslant n_0 \leqslant k $, $N_k \leqslant m\leqslant n$, and  $x\in X$ with $\|x\|=1$ we have that
\begin{equation*}
    \begin{split}
     \|P_{[m_0,n_0]} P_{[m,n]} x\|  & =   \Big\|\sum_{j=m_0}^{n_0}f_j(\sum_{i=m}^n f_i(x)x_i)x_j\Big\|\\
     & \leqslant  k \sup_{1\leqslant j\leqslant k} \Big\|f_j(\sum_{i=m}^n f_i(x)x_i)x_j\Big\|\\
     & \leqslant k \sup_{1\leqslant j\leqslant k}\Big\|\sum_{i=m}^n f_i(x_j)f_i\Big\|\|x\|\|x_j\| < \varepsilon 
    \end{split}
\end{equation*}
The claim follows.
\end{proof}

\begin{remark}\label{R:sub}
Note that Proposition \ref{P:N_k} still holds if we replace $(2^{-k})_{k=1}^\infty$ with any positive sequence which converges to $0$. Moreover, if $(N_k)_{k=1}^\infty$ satisfies \eqref{killstuff} and $(M_k)_{k=1}^\infty$  is any increasing sequence of natural numbers with $N_k\leqslant M_k$ for all $k\in\N$ then $(M_k)_{k=1}^\infty$ also satisfies \eqref{killstuff}. 
\end{remark}

We now define a norm on $c_{00}$ which we will later prove gives a shrinking associated basis.

\begin{definition}
Let $(x_i,f_i)_{i=1}^\infty$ be a shrinking Schauder frame and let  $(N_k)_{k=1}^\infty$ satisfy the hypothesis of Proposition \ref{P:N_k}. We let $(z_i)_{i=1}^\infty$ denote the unit vector basis of $c_{00}$ and consider the following norm for $\sum a_i z_i\in c_{00}$.
\begin{equation}\label{E:def}
     \Big\|\sum a_i z_i\Big\|_{(N_k)}=\sup_{m\leqslant n}\Big\|\sum_{m\leqslant i\leqslant n} a_i x_i\Big\| \vee 
    \sup_{\substack{m_0\leqslant n_0\leq k\\ N_k \leqslant m\leqslant n}} 2^k \Big\|P_{[m_0,n_0]}  \sum_{m\leqslant i\leqslant n} a_i x_i\Big\|.
\end{equation}
Let $Z_{(N_k)}$ denote the completion of $c_{00}$ under this norm. For $z \in Z_{(N_k)}$ and $k'\in\N$ it will be convenient to denote the second part of \eqref{E:def} as
\begin{equation}
    \|z\|_{k'} := \sup_{\substack{m_0\leqslant n_0\leq k'\\ N_{k'} \leqslant m\leqslant n}} 2^{k'} \|P_{[m_0,n_0]} S_{[m,n]}z\|.
\end{equation}
\end{definition}

Proposition \ref{P:N_k} gives a condition satisfied by each shrinking Schauder frame. The idea behind the definition of the norm above is to force the associated space to satisfy some version of this condition. The goal then is to show that satisfying this condition is sufficient to establish that the associated basis is shrinking.

\begin{remark}
A slight weakening of the norm $\|\cdot\|_{(N_k)}$ was introduced in \cite{BFL-Fund} where the authors prove that the basis $(z_i)$ is strongly shrinking relative to $(x_i,f_i)$, which is a weaker condition than shrinking.  
\end{remark}

\begin{theorem}\label{T:main}
Let $(x_i,f_i)_{i=1}^\infty$ be a shrinking frame for a Banach space $X$ and let $(N_k)_{k=1}^\infty$ satisfy Proposition \ref{P:N_k}. Then $Z_{(N_k)}$ is an associated space for $(x_i,f_i)_{i=1}^\infty$ and $(z_i)_{i=1}^\infty$ is a shrinking basis for the space $Z_{(N_k)}$. 
\label{main}
\end{theorem}

\begin{proof}

Assuming $(N_k)_{k=1}^\infty$ satisfies Proposition \ref{P:N_k}, we will first show that $Z_{(N_k)}$ is as an associated space to the frame $(x_i,f_i)_{i=1}^\infty$. This is the only place in the proof we use Proposition \ref{P:N_k}. Let us see that the analysis operator $T:X \to Z_{(N_k)}$ satisfies $\|T\|\leq C$ where $C:=\sup_{m\leqslant n}\|P_{[m,n]}\|$ is the frame constant of $(x_i,f_i)$. Let $x\in X$.  Then, the first part of $\|T x\|_{(N_k)}=\|\sum f_i(x)z_i\|_{(N_k)}$ in \eqref{E:def} satisfies 
$$\sup_{m\leq n}\Big\|\sum_{i=m}^n f_i(x)x_i\Big\|=\sup_{m\leq n}\|P_{[m,n]}x\|\leq C\|x\|.$$
We now fix $k\in \mathbb{N}$, $m_0\leqslant n_0\leqslant k$, and $N_k \leqslant m\leqslant n$. Then by Proposition \ref{P:N_k} we have  that the second part of  \eqref{E:def} satisfies 
$$2^k\Big\|P_{[m_0,n_0]}S_{[m,n]}\sum f_i(x)z_i\Big\|=2^k\Big\|P_{[m_0,n_0]}P_{[m,n]}\sum f_i(x)x_i\Big\|\leqslant \|x\|. $$
Thus, we have that $\|T x\|_{(N_k)}\leq C\|x\|$ and hence $\|T\|\leq C$.
The synthesis operator $S:Z_{(N_k)}\to X$ is bounded, since it is bounded on $S:Z_{\min}\to X$ and $\|z\|_{\min}\leqslant \|z\|_{(N_k)}$.  Thus, $Z_{(N_k)}$ is an associated space to $(x_i,f_i)_{i=1}^\infty$.

 Let $(y_i)$ be a normalized block sequence of $(z_i)$ in $Z_{(N_k)}$. In order to show that $(z_i)$ is shrinking it suffices to show by Lemma \ref{L:w} that there is a subsequence of $(y_i)$ which is weakly null. We claim that we may pass to a subsequence of $(y_i)$ and find an increasing sequence $(k_i)$ in $\mathbb{N}$ so that for all $i\in\N$ we have that
\begin{itemize}
    \item[(i)] $y_i \in \Span_{N_{k_i}\leqslant j\leqslant k_{i+1}} (z_j)$,
    \item[(ii)] $\displaystyle \|P_{[m,n]}Sy_i\| \leqslant 2^{-k_i}$ for all $m,n\in\N$ with  $k_{i+1} \leqslant m \leqslant n$.
\end{itemize}
Indeed, (i) is easily obtained as $(y_i)$ is a block sequence of $(z_i)$.  We may obtain (ii) by choosing $k_{i+1}$ sufficiently large as $\sum_{j\in\N} f_j (Sy_i)x_j$ is convergent for all $i\in\N$.  The following additional properties are implied by (i) and (ii).
\begin{enumerate}
    \item[(iii)] $\|P_{[m_0,n_0]}S_{[m,n]}y_i\|\leqslant 2^{-k_i}$ for $m_0\leqslant n_0\leqslant k_i$ and $m\leqslant n$,
    \item[(iv)] For each $i \in \mathbb{N}$, $\|P_{[k_i,k_{i+1})} Sy_i\| \geqslant \|Sy_i\|- 2^{-k_i+1}$,
    \item[(v)] For $i \not= j$ in $\mathbb{N}$,  $\|P_{[k_j,k_{j+1})} Sy_i\| \leqslant  2^{-k_i}$. 
\end{enumerate}
Item (iii) follows from (i) and the fact that $\|y_i\|_{k_i}\leqslant \|y_i\|_{(N_k)} = 1$.
Item (v) follows from (iii) if $j<i$ and follows from (ii) if $j>i$. Item (iv) is a consequence of (ii) and (iii) as
$$\|P_{[k_i,k_{i+1})} Sy_i\|\geqslant \|Sy_i\|-\|P_{[1,k_{i})} Sy_i\|-\|P_{[k_{i+1},\infty)} Sy_i\|\geqslant \|Sy_i\|- 2^{-k_i+1}.$$
Before dividing the proof into two cases, we fix $(a_i)\in c_{00}$ and $k\in\N$ and will show that 
\begin{equation} \label{upperc0}
    \big\|\sum a_i y_i\big\|_k=\sup_{\substack{m_0<n_0\leq k\\ N_k \leqslant m<n}} 2^k \big\|P_{[m_0,n_0]} S_{[m,n]}\sum a_i y_i\big\|\leqslant 2 \sup |a_i|.
\end{equation}
Let $m_0\leqslant n_0\leqslant k$ and $N_k \leqslant m\leqslant  n$. Let $i_0$ be the least integer such that $m\leqslant k_{i_0+1}$. By (i), we have that $S_{[m,n]}y_i=0$ for all $i<i_0$.  Since $\|y_{i_0}\|=1$,  \eqref{E:def} implies that $2^k\|P_{[m_0,n_0]}S_{[m,n]} a_{i_0}y_{i_0}\| \leqslant |a_{i_0}|$.
By (iii),
\begin{equation}
\sum_{i=i_0+1}^\infty 2^k\|P_{[m_0,n_0]}S_{[m,n]} a_iy_i\|\leqslant 2^{k} \sum_{i=i_0+1} \frac{1}{2^{k_i}}|a_i| \leqslant \sup|a_i|.  \label{sums}
\end{equation}
Thus, we have that 
\begin{align*}
     2^k \|P_{[m_0,n_0]} S_{[m,n]}\sum a_{i} y_{i}\|&\leqslant 2^k\|P_{[m_0,n_0]} S_{[m,n]} a_{i_0} y_{i_0}\|+\sum_{i=i_0+1}^\infty 2^k\|P_{[m_0,n_0]}S_{[m,n]} a_iy_i\|\\
     &\leqslant 2\sup|a_i|.
\end{align*}
This proves \eqref{upperc0}.  We now pass to a further subsequence of $(y_i)$ such that exactly one of the following holds.
\begin{itemize}
   \item[(vi a.)] $\|S y_i\|\leqslant 2^{-i}$ for each $i \in \mathbb{N}$.
    \item[(vi b.)] For some $c>2^{-k_1+3}$, we have $c\leqslant \|Sy_i\|\leqslant 1$.
\end{itemize}
Assume (vi a.) holds. Fix $(a_i)\in c_{00}$. Let $m\leqslant n$ and let $i_0$ be the least integer such that $m\leqslant \max\supp\,y_{i_0}$ and let $i_1$ be the greatest integer such that $n\geqslant \min\supp y_{i_1}$. We have the following.
\begin{align*}
    \|S_{[m,n]}\sum a_i y_i\|&=\|S_{[m,n]} \sum_{i=i_0}^{i_1} a_i y_i\|\\
    &\leq \|S_{[m,n]} a_{i_0} y_{i_0}\|+\sum_{i=i_0+1}^{i_1-1} \|S a_i y_i\|+\|S_{[m,n]} a_{i_1} y_{i_1}\|\\
    &\leq    |a_{i_0}| + \sum_{i=i_0+1}^{i_1-1} 2^{-i}|a_i|+|a_{i_1}|\hspace{1cm}\textrm{ by (vi),}\\
    &\leq 3\sup_{i\in\N} |a_i|
\end{align*}
Combining this with (\ref{upperc0}) gives that $(y_i)$ is $3$-dominated by the unit vector basis of $c_0$ and therefore $(y_i)$ is weakly null, which completes the proof for this case.

We now assume item (vi b.) holds. We will prove that in this case that $(Sy_i)$ and $(y_i)$ are equivalent basic sequences and that $(Sy_i)$ is weakly null.  Let $C$ be the frame constant of $(x_i,f_i)$.  We first prove that $(Sy_i)$ is $2C$-basic.

Let $(a_i) \in c_{00}$ and let $j\in\N$ be such that $|a_{j}|=\sup |a_i|$. We have that 
\begin{align*}
    C\big\|\sum a_i Sy_i\big\| &  \geqslant \|\sum a_i P_{[k_{j},k_{j+1})} S y_i\| \\
    & \geqslant  \big\|a_j P_{[k_j,k_{j+1})} S y_j\big\| -  \sum_{i\not=j} \big\|a_i P_{[k_j,k_{j+1})} S y_i\big\| \\
    & \geqslant (|a_j|\|S y_i\|-|a_j|2^{-k_j+1}) - \sum_{i\not=j} |a_i|2^{-k_i}\hspace{.5cm}\textrm{ by  (iv), and (v)}\\
    & \geqslant |a_j|c-|a_j|2^{-k_1+2}\hspace{1cm}\textrm{ by  (vi b.) }\\
    &\geqslant c 2^{-1} |a_j|\hspace{2cm}\textrm{ as $c>2^{-k_1+3}$ by  (vi b.). }
\end{align*}
Thus, we have that 
\begin{equation}
    C\|\sum a_i S y_i\|\geqslant c2^{-1} \sup|a_i|. \label{more lower}
\end{equation}
We now fix $M\in\N$ and consider the following partial sum.
\begin{align*}
    \Big\|\sum_{i=1}^M & a_i Sy_i\Big\|  \leqslant \Big\|P_{[1,k_{M+1})}\sum_{i=1}^M a_i Sy_i\Big\| + \Big\|P_{[k_{M+1},\infty)}\sum_{i=1}^M a_i Sy_i\Big\|  \\
    & \leqslant \big\|P_{[1, k_{M+1})}\sum a_i Sy_i\big\| + \Big\|P_{[1,k_{M+1})}\sum_{i=M+1}^\infty a_i Sy_i\Big\|
    +\Big\|P_{[k_{M+1},\infty)}\sum_{i=1}^M a_i Sy_i\Big\|\\
    & \leqslant C \big\|\sum a_i Sy_i\big\| +\sum_{i=M+1}^\infty |a_i|2^{-k_i}+ \sum_{i=1}^M |a_i|2^{-k_i}\hspace{.5cm}\textrm{ by  (iii) and (ii)}\\
    & \leqslant C \big\|\sum a_i Sy_i\big\| + \sup |a_i| 2^{-k_1+1}\\
    & \leqslant C \big\|\sum a_i Sy_i\big\| +  \sup |a_i| c2^{-2}\hspace{1cm}\textrm{ as $c>2^{-k_1+3}$ by  (vi b.)}\\
    &\leqslant (C + 2^{-1}C) \big\|\sum a_i Sy_i\big\|\leqslant 2C \big\|\sum a_i Sy_i\big\|\hspace{1cm}\textrm{ by \eqref{more lower}}.
\end{align*}
This proves that $(Sy_i)$ is $2C$-basic.

Since $S$ is a bounded linear operator,  $(y_i)$  dominates  $(Sy_i)$.  We now prove that $(Sy_i)$ is equivalent to $(y_i)$ by proving that $(Sy_i)$  dominates $(y_i)$.
Fix $(a_i)\in c_{00}$. Let $j_0\in\N$ with $|a_{j_0}|=\max_i |a_i|$ and let $I_{j_0}\subset\mathbb{N}$ be the smallest interval containing $\supp y_{j_0}$. Thus by (vi b.),
\begin{equation}
\sup_{m\leq n}\big\|S_{[m,n]} \sum a_i y_i\big\|\geqslant \big\|S_{I_{j_0}} \sum a_i y_i\big\|=|a_{j}|\|S y_{j}\|\geqslant c\sup|a_i|.\label{at least c0}
\end{equation}
We now have that
\begin{align*}
        \big\|\sum a_i y_i\big\|& =\sup_{m\leq n}\big\|S_{[m,n]} \sum a_i y_i\big\| \vee \sup_k \big\|\sum a_i y_i\big\|_k \\
& \leqslant \sup_{m\leq n}\big\|S_{[m,n]} \sum a_i y_i\big\| \vee 2 \sup |a_i|\hspace{1cm}\textrm{ by \eqref{upperc0}},\\
&  \leqslant \sup_{m\leq n}\big\|S_{[m,n]} \sum a_i y_i\big\| \vee \frac{2}{c} \sup_{m\leq n}\big\|S_{[m,n]} \sum a_i y_i\big\| \hspace{1cm}\textrm{ by \eqref{at least c0}},\\
& =  \frac{2}{c} \sup_{m\leq n}\big\|S_{[m,n]} \sum a_i y_i\big\|
\end{align*}
Therefore, to prove that $(S y_i)_{i=1}^\infty$ dominates $(y_i)_{i=1}^\infty$ it will suffice to prove that for fixed $m<n$ we have that 
\begin{equation}\label{E:dom}
    \big\|S_{[m,n]}\sum a_i y_i\big\| \leqslant 2C(1+2c^{-1}) \big\|\sum a_i S y_i\big\|.
\end{equation}  As, $(y_i)$ is a block sequence of $(z_i)$, there exists $j_1\leqslant j_2$ so that 
$$S_{[m,n]}\sum_i a_i y_i =  \sum_{i=j_1}^{j_2} a_j Sy_j - S_{[1,m)}a_{j_1}y_{j_1} - S_{[n,\infty)}a_{j_2}y_{j_2}  $$
By taking the norm of both sides we now have that
\begin{align*}
    \big\|S_{[m,n]}\sum a_i y_i\big\| &\leqslant  \Big\|\sum_{i=j_1}^{j_2} a_j Sy_j\Big\| + \|S_{[1,m)}a_{j_1}y_{j_1}\|+ \|S_{[n,\infty)}a_{j_2}y_{j_2}\|\\
     &\leqslant  \Big\|\sum_{i=j_1}^{j_2} a_j Sy_j\Big\| + |a_{j_1}|+ |a_{j_2}|\\
    &\leqslant \Big\|\sum_{i=j_1}^{j_2} a_j Sy_j\Big\| + c^{-1}\|a_{j_1} S y_{j_1} \|+ c^{-1}\|a_{j_2} S y_{j_2} \|\\
    &\leqslant 2C(1+2c^{-1}) \big\|\sum a_i S y_i\big\|\hspace{1cm}\textrm{as $(S y_i)$ is $2C$-basic.}
\end{align*}
This proves \eqref{E:dom} and hence $(Sy_i)$ and $(y_i)$ are equivalent basic sequences.  
All that remains is to prove that $(S y_j)$ is weakly null. Let $f\in X^*$ be some functional.
\begin{align*}
\lim_{j\rightarrow\infty} |f(S y_j)|&=  \lim_{n\rightarrow\infty}\lim_{j\rightarrow\infty} |f(P_{[n,\infty)}Sy_j)|\hspace{1cm}\textrm{by (iii),}\\
&\leq \lim_{n\rightarrow\infty}\lim_{j\rightarrow\infty} \|f\circ P_{[n,\infty)}\| \| S y_j\|\\
&=0\hspace{2cm}\textrm{ as the Schauder frame $(x_i,f_i)$ is shrinking.}
\end{align*}
Thus, $(Sy_j)$ is weakly null which implies that $(y_j)$ is weakly null as they are equivalent basic sequences.  Hence, $(z_i)$ is a shrinking basis as every normalized block sequence is weakly null.
\end{proof}

The proof of the above theorem admits the following corollary.

\begin{corollary}
Let $(x_i,f_i)$ be a shrinking frame for a Banach space $X$ and let $(N_k)$ satisfy Proposition \ref{P:N_k}. Let $(y_i)$ be a normalized block sequence  in $Z_{(N_k)}$.
\begin{enumerate}
    \item If there is a subsequence $(y'_i)$ of $(y_i)$ so that
    $Sy'_i\to 0$, then there is a further subsequence of $(y'_i)$ that is equivalent the the unit vector basis of $c_0$.
    \item If there is no subsequence $(y'_i)$ of $(y_i)$ so that 
    $Sy'_i\to0$, then there is a subsequence  $(y'_i)$ of $(y_i)$ so that $(y'_i)$ is equivalent to $(Sy'_i)$.  
\end{enumerate}
\label{dichotomy}
\end{corollary}
Note that the example given in Section 2 shows that Corollary \ref{dichotomy} is false for the minimal associated $Z_{min}$.  In that example, $(z_{2j-1})_{j=1}^\infty$ is a normalized block sequence in $Z_{min}$ with $\|S z_{2j-1}\|=1$ for all $j\in\N$, but $(z_{2j-1})_{j=1}^\infty$ has no subsequence which is equivalent to a sequence in $X$.

\begin{remark}
Let $X$ have a Schauder frame $(x_i,f_i)$ with associated space $Z$,  analysis operator $T:X \to Z$ and synthesis operator $S: Z \to X$. The following are fundamental properties of Schauder frames.
\begin{enumerate}
    \item $X$ is isomorphic to $TX$, which is a complemented subspace of $Z$.
    \item $TS:Z\rightarrow Z$ is a projection of $Z$ onto $TX$.
    \item $Z/TX$ is isomorphic to the range of $I_Z-TS$, where $I_Z$ is the identity operator on $Z$.
\end{enumerate}
\label{I-TS}
\end{remark}

\begin{proposition}
Let $(x_i,f_i)$ be a shrinking frame for a Banach space $X$ and let $(N_k)$ satisfy Proposition \ref{P:N_k}. Then $Z_{(N_k)}/TX$ is $c_0$ saturated.  That is, every infinite dimensional subspace of $Z_{(N_k)}/TX$ contains a further subspace which is isomorphic to $c_0$.
\end{proposition}

\begin{proof}
Using Remark \ref{I-TS}, $Z_{(N_k)}/TX$ is isomorphic 
to the range of $I_Z-TS$ in $Z_{(N_k)}$. 
Let $Y$ be an infinite dimensional subspace of the range of $I_Z-TS$.  There exists a normalized block sequence $(y_i)$ in $Z_{(N_k)}$ and a sequence $(w_i)$ in $Y$ so that $\|y_i-w_i\|\to 0$.  After passing to a subsequence, we may assume that $(y_i)$ and $(w_i)$ are equivalent basic sequences.
 As $Y$ is contained in the range of $I_Z-TS$ and $TS$ is a projection operator, we have that $TSw_i=0$ for all $i \in \mathbb{N}$. Therefore $TSy_i\to 0$. Since $T$ is a an isomorphic embedding, we have that $Sy_i \to 0$. Therefore we are in the first alternative of Corollary \ref{dichotomy} and so $(y_i)$ has a subsequence equivalent to the unit vector basis of $c_0$. 
\end{proof}

The above can be compared to the result of Liu-Zheng \cite{LZ-JFAA} in which the authors prove that if $Z_{min}$ is the minimal associated space for a Schauder frame of $X$, then $Z_{min}/X$ contains an isomorphic copy of $c_0$ if and only if $Z_{min}/X$ is infinite dimensional.

\section{Comparing Associated Spaces with Shrinking Bases}\label{S:compare}

Our next result illustrates that associated spaces of the form $Z_{(N_k)}$ are a minimal collection, with respect to domination, among associated spaces with shrinking bases.

\begin{theorem}
Suppose that $(x_j,f_j)$ is a shrinking Schauder frame for $X$ and that $W$ is an associated space of  $(x_j, f_j)$ with a shrinking associated basis $(w_j)$.  Then there exists $(N_k)$ so that, the basis $(z_j)$ of $Z_{(N_k)}$ is dominated by $(w_j)$.\label{dominated by}
\end{theorem}

We isolate the following remark.

\begin{remark}
Let $\mathcal{A}$ be a finite collection of finite rank operators on a Banach space $W$ with a shrinking basis $(w_j)$. Then,
$$\lim_{n\to \infty} \sup_{A\in \mathcal{A}} \|A \circ R_{[n,\infty)}\|=0 \text{ where }R_{[n,\infty)}(\sum a_i w_i)= \sum_{i\geqslant n}a_iw_i.$$
\label{finite rankers}
\end{remark}

\begin{proof}[Proof of Theorem \ref{dominated by}]
Fix $(x_i,f_i)$, $(w_j)$ and $W$ as in the statement of the theorem and $R_{[n,\infty)}$ as in Remark \ref{finite rankers}.  Since $(w_i)$ dominates the minimal associated basis of $(x_i,f_i)$ there exists  $K\geq 1$ so that 
\begin{equation}
    \sup_{m\leqslant n}\|\sum_{i=m}^n a_i x_i\|\leqslant K\|\sum a_i w_i\|\hspace{1cm}\textrm{ for all }(a_i)\in c_{00}.
\end{equation}
 Let $S_W:W\to X$ be the synthesis operator. For each $k\in\N$, the set $\{P_{[m_0,n_0]}\circ S_W : m_0<n_0\leqslant k\}$
is a finite collection of finite rank operators on $W$. By the previous remark, for each $k\in\N$ there exists $N_k\in\N$ so that $\|(P_{[m_0,n_0]}\circ S_W)\circ R_{[m,n]}\|<2^{-k}$ for all $m_0<n_0\leqslant k$ and $N_k\leqslant m\leqslant n$.
Thus we have for all $(a_i)\in c_{00}$ that

\begin{align*}
K\Big\|\sum a_i w_i\Big\|&\geqslant \sup_{m\leqslant n}\Big\|\sum_{m\leqslant i\leqslant n} a_i x_i\Big\| \vee \sup_{\substack{m_0\leqslant n_0\leq k\\ N_k \leqslant m\leqslant n}} 2^k \Big\|P_{[m_0,n_0]}  \sum_{m\leqslant i\leqslant n} a_i x_i\Big\|\\
&=  \Big\|\sum a_i z_i\Big\|_{(N_k)}
\end{align*}
Hence, we have that $(w_i)$ $K$-dominates the basis $(z_i)$ of $Z_{(N_k)}$.
\end{proof}

 Let $(x_j,f_j)$ be a shrinking Schauder frame and let $(N_j)$ be an increasing sequence of natural numbers which satisfies Proposition \ref{P:N_k}. If $(k_j)$ is an increasing subsequence of natural numbers then we denote $Z_{(k_j),(N_j)}$ to be the Banach space with basis $(z^{(k_j)}_i)_{i=1}^\infty$ which is the completion of $c_{00}$ under the norm:
 \begin{equation}
\big\|\sum a_i z^{(k_j)}_i\big\|_{(k_j),(N_j)}= \sup_{m\leq n}\big\|\sum_{m\leqslant i \leqslant n} a_i x_i\big\| \vee \sup_{j} \big\|\sum a_i z_i\big\|_{k_j}.
\end{equation}
Recall that for $z\in Z_{(N_j)}$ and $j\in \N$,
\begin{equation}
    \|z\|_{k_j} := \sup_{\substack{m_0\leqslant n_0\leq k_j\\ N_{k_j} \leqslant m\leqslant n}} 2^{k_j} \big\|P_{[m_0,n_0]} S_{[m,n]}z\big\|.
\end{equation}
It follows from Remark \ref{R:sub} and Theorem \ref{main} that $Z_{(k_j),(N_j)}$ is an associated space of $(x_i,f_i)_{i=1}^\infty$ and that $(z^{(k_j)}_i)_{i=1}^\infty$ is a shrinking basis for $Z_{(k_j),(N_j)}$.

\begin{theorem} \label{lots of them}
Suppose that $(x_j,f_j)$ is a shrinking Schauder frame for a Banach space $X$ so that the minimal associated basis is not shrinking.  Then for any sequence $(N_k)$ satisfying Proposition \ref{P:N_k}, there exists an increasing sequence $(k_i)_{i=1}^\infty$ so that for all infinite subsets $L,M\subset \mathbb{N}$, the following are equivalent. 
\begin{enumerate}
    \item $(z^{(k_i)_{i \in M}}_j)$ dominates $(z^{(k_i)_{i \in L}}_j)$.
    \item $L\setminus M$ is finite.
\end{enumerate}
\end{theorem}

Before proving Theorem \ref{lots of them} we state and prove the following corollary.

\begin{corollary}
Let $(x_j,f_j)$ be a shrinking Schauder frame so that the minimal associated basis is not shrinking. Then for each  shrinking associated basis $(w_j)$ there are increasing sequences of natural numbers $(k_i)$ and $(N_i)$ and a set of increasing sequences of natural numbers $(M_\alpha)_{\alpha\in\Delta}$ with $\Delta$ having cardinality the continuum so that
\begin{enumerate}
    \item For each $\alpha \in \Delta$, the basis 
    $(z^{(k_i)_{i\in M_\alpha}}_j)$ of $Z_{(k_i)_{i\in M_\alpha},(N_i)}$ is a shrinking associated basis of $(x_j,f_j)$ which is dominated by $(w_j)$.
    \item For $\alpha \not= \beta$ in $\Delta$, the bases $(z^{(k_i)_{i\in M_\alpha}}_j)$ and $(z^{(k_i)_{i\in M_\beta}}_j)$ are incomparable.
\end{enumerate}
\end{corollary}

\begin{proof}
Fix a shrinking associated basis $(w_j)$. By Theorem \ref{dominated by}, there is a sequence $(N_i)$ so that the basis $(z_j)$ of $Z_{(N_i)}$ is dominated by $(w_j)$. Let $(k_i)$ be a sequence which satisfies the conclusion of Theorem \ref{lots of them}. For every infinite $M\subseteq\N$, the basis $(z^{(k_i)_{i\in M}}_j)$ of $Z_{(k_i)_{i\in M},(N_i)}$ is a shrinking associated basis of $(x_i,f_i)$ which is dominated by $(z_j)$. Hence, $(z^{(k_i)_{i\in M}}_j)$  is also dominated by $(w_j)$.
 Let  $(M_\alpha)_{\alpha \in \Delta}$ be a collection of infinite subsets of $\mathbb{N}$ with cardinality the continuum so that $M_\alpha \cap M_{\beta}$ is finite for all $\alpha \not= \beta$. This is called a collection of almost disjoint sets and is known to exist.  In particular, $M_\alpha \setminus M_{\beta}$ is infinite for all $\alpha\neq\beta$.  
 By Theorem \ref{lots of them}, $(z^{(k_i)_{i \in M_\alpha}}_j)$ and $(z^{(k_i)_{i \in M_\beta}}_j)$ are incomparable  basic sequences for all $\alpha\neq\beta$. 
\end{proof}

\begin{proof}[Proof of Theorem \ref{lots of them}]
Let $(x_i,f_i)$ be a shrinking Schauder frame for $X$ so that the minimum associated basis is not shrinking. Let $(N_k)$ be an increasing sequence of natural numbers which satisfies Proposition \ref{P:N_k}. By Theorem \ref{T:main} the basis $(z_i)$  of $Z_{(N_k)}$ is a shrinking associated basis  of $(x_i,f_i)$ and, moreover for each sequence $(k_j)$, the basis $(z_i^{(k_j)})$ of $Z_{(k_j),(N_{k_j})}$  is a shrinking associated space of $(x_i,f_i)$. 

As the minimal associated basis of $Z_{\min}$ is not shrinking, there exists a normalized block sequence $(y_n)$ in $Z_{\min}$ which is $\alpha$-$\ell_1^+$ for some $\alpha>0$.  The $Z_{(N_k)}$ norm $1$-dominates the $Z_{\min}$ norm.  Hence, for all non-negative scalars $(a_n)$ we have that 
$$\alpha\sum a_n\leq \big\|\sum a_n y_n\big\|_{Z_{\min}}\leq \big\|\sum a_n y_n\big\|_{(N_i)}.
$$
As $(z_i)$ is a shrinking basis for $Z_{(N_k)}$, every bounded block sequence converges weakly to $0$ and is hence not $\ell_1^+$.   Thus, $(y_n)$ cannot be norm bounded as a sequence in $Z_{(N_i)}$.  After passing to a subsequence, we assume that $\|y_n\|_{(N_i)}\geq 2^n$ for all $n\in\N$.

Let $C>0$ be the frame constant of $(x_i,f_i)$.  Thus, for all $k\in\N$, and $z\in Z_{min}$ we have that $\|z\|_{k}\leq C2^k\|z\|_{Z_{min}}$.  As $\|y_n\|_{Z_{min}}=1$ for all $n\in\N$ and $(y_n)$ is unbounded in $Z_{(N_i)}$, after passing to a subsequence of $(y_n)$ we may assume that there exists a sequence $(k_i)$ so that for all $j\in\N$, $\|y_{j+1}\|_{k_{j+1}}\geq 2^{2k_j}$ and $\supp(y_{t_{j+1}})\subseteq [N_{k_j}, N_{k_{j+2}})$.  We now assume that $L \setminus M$ is infinite and will prove that the basis $(z_j^{(k_i)_{i\in M}})$ does not dominate the basis $(z_j^{(k_i)_{i\in L}})$.  Let $d\in L\setminus M$ with $d>1$.  We have that the following holds
\begin{align*}
    \|y_{d}\|_{(k_i)_{i\in M},(N_i)}&=\|y_{d}\|_{Z_{min}}\vee \sup_{i\in M}\|y_{d}\|_{k_i}\\
     &\leqslant \|y_{d}\|_{Z_{min}}\vee \|y_{d}\|_{k_{d-1}}\hspace{1cm}\textrm{ as }\supp(y_{d})\subseteq [N_{k_{d-1}}, N_{k_{d+1}})\\
    &\leqslant  \|y_{d}\|_{Z_{min}}\vee C2^{k_{d-1}}  \|y_{d}\|_{Z_{min}}\\
    &=C 2^{k_{d-1}}\\
    &\leqslant C 2^{-k_{d-1}}\|y_{d}\|_{(k_i)_{i\in L},(N_i)}\hspace{1cm}\textrm{ as }\|y_{d}\|_{k_{d}}\geq 2^{2k_{d-1}}
\end{align*}
 Thus, the basis $(z_j^{(k_i)_{i\in M}})$ does not dominate the basis $(z_j^{(k_i)_{i\in L}})$ as $L\setminus M$ is infinite.  We now assume that $L\setminus M$ is finite.   Let $m'$ be the least element of $M$ so that $L \cap [m',\infty) \subseteq M $.  Let $z\in \textrm{span}_{i\geq N_{k_{m'}}} z_i$.
 Thus, $\|z\|_{k_i}\leq \|z\|_{k_{m'}}$ for all $i\leq m'$.  We have that
\begin{align*}
    \|z\|_{(k_i)_{i\in L},(N_i)}&=\|z\|_{Z_{min}}\vee \sup_{i\in L}\|z\|_{k_i}\\
    &\leq \|z\|_{Z_{min}}\vee \sup_{i\in M,i\geq m'}\|z\|_{k_i}\\
    &=\|z\|_{(k_i)_{i\in M},(N_i)}
\end{align*}
 Thus, $(z_j^{(k_i)_{i\in M}})_{j=N_{k_{m'}}}^\infty$ $1$-dominates the basis $(z_j^{(k_i)_{i\in L}})_{j=N_{k_{m'}}}^\infty$.  This implies that   $(z_j^{(k_i)_{i\in M}})_{j=1}^\infty$ $K$-dominates $(z_j^{(k_i)_{i\in L}})_{j=1}^\infty$ for some $K\geq 1$.
 \end{proof}

\section{Shrinking Bounded Approximation Property}

A separable Banach space $X$ has the Bounded Approximation Property BAP if there is a sequence $(B_n)$ of finite rank operators on $X$ so that $\lim_n\|x-B_nx\|=0$ for all $x \in X$. The uniform boundedness principle implies that whenever this condition holds there is a $\lambda>0$ with $\|B_n\|\leqslant \lambda$ for each $n$. A space with this property for $\lambda$ is said to have the $\lambda$-AP. Moreover, setting $A_1=B_1$ and $A_n=B_n-B_{n-1}$ for $n>1$ we can replace $B_n$ with $\sum_{i=1}^n A_i$. We note that the {\em definition} of the $\lambda$-AP is that there for each $\varepsilon>0$ and compact set $K$ in $X$ there is finite rank operator $T$ with $\|T\|\leqslant \lambda$ so that $\|x-Tx\|\leqslant \varepsilon$ for all $x \in K$. 

To mirror the definition of shrinking basis, one may wish to define a space $X$ to have the shrinking-BAP if there are finite rank operators $(B_n)$ on $X$ so that $\lim_n\|x-B_nx\|=0$ for all $x \in X$ and $\lim_n\|f-B_n^*f\|=0$ for all $f \in X^*$. That is, the operators in the space approximating the identity also have the property that their dual operator approximate the identity. This is analogous to: A basis $(x_n)$ is shrinking if and only if the coordinate functionals $(x_n^*)$ form a basis for $X^*$. 

The above definition of shrinking-BAP is formally stronger than simply $X^*$ having the BAP and has been isolated before under the name duality-BAP \cite[page 288]{Casazza-Handbook}. The surprising fact that these notions are equivalent  is the content of the following proposition \cite[Proposition 3.5]{Casazza-Handbook}.

\begin{theorem}
A space $X$ has the shrinking BAP if and only if $X^*$ has the BAP. 
\label{equivalent}
\end{theorem}

In fact a lot more is known: A dual space $X^*$ has the AP if and only if $X$ and $X^*$ have the $1$-AP (i.e. the metric approximation property). Another result related to the current work is the following \cite[Theorem 4.9]{Casazza-Handbook}.

\begin{theorem}
Let $X$ be a Banach space with separable dual. Then $X^*$ has the BAP if and only if $X$ embeds complementably in a Banach space with a shrinking basis. 
\label{Zippincomp}
\end{theorem}

Theorem \ref{Zippincomp} is a complemented version of Zippin's theorem \cite{Zi-TAMS} stating that every Banach space with a separable dual embeds into a space with a shrinking basis. It is also a refinement of the aforementioned theorem of Pe\l czynski and Johnson-Rosethal-Zippin stating that every space with the BAP embeds complementably into a space with a basis. The proof of Theorem \ref{Zippincomp} (as stated \cite{Casazza-Handbook}) follows the results in \cite{JRZ-Israel}. Here the authors show that if $X^*$ has the BAP then $X\oplus C_p$ has a shrinking basis where $C_p$ is the $\ell_p$ sum of finite dimensional spaces $(E_n)$ which are dense (with the Banach Mazur distance) in the space of all finite dimensional spaces. 

We present an alternative proof Theorem \ref{Zippincomp} using the language of frames and modifying the proof of Pe{\l}czy{\'n}ski\cite{Pel-BAP}. The technique we employ for this proof is also used by Mujica and Vieira in \cite{MV-Studia} to give a quantitative improvement of Pe{\l}czy{\'n}ski's theorem. 

\begin{theorem}
Let $X$ be a Banach space. Then $X^*$ has the BAP only if and only if $X$ has a shrinking Schauder frame. \label{shrinking-BAP}
\end{theorem}

Note that Theorem \ref{Zippincomp} follows immediately from combining Theorem \ref{shrinking-BAP} with Theorem \ref{main}.

\begin{proof}[Proof of Theorem \ref{shrinking-BAP}]  If $X$ has a shrinking Schauder frame $(x_j,f_j)$ then $(f_j,x_j)$ is a Schauder frame of $X^*$ and hence $X^*$ has the BAP. Before proving the reverse direction  we prove the following finite dimensional result.  Let $X$ be a Banach space and let $A:X\rightarrow X$ be a finite rank operator.  Let $d$ be the rank of $A$ and let $m \in \mathbb{N}$. Then there exists  $(x_j,f_j)_{j=1}^{md}\subseteq X\times X^*$  such that the sequence of rank one operators $(f_j\otimes x_j)_{j=1}^{md}$ satisfies

\begin{equation}
\begin{split}\label{E:Pel}
&\sum_{j=1}^{qd} f_j\otimes x_j\!=\frac{q}{m}\!A\hspace{.5cm}\textrm{ for all $0\leqslant q\leqslant m$,}\\
\Big\|&\sum_{j=qd+1}^{qd+r} f_j\otimes x_j\Big\|\!\leqslant\! \frac{r}{m}\|A\|\hspace{.5cm}\textrm{ for all $0\leqslant q< m$ and $0\leqslant r \leqslant d$.}
\end{split}
\end{equation}
Indeed, let $(e_i)_{i=1}^d$ be an Auerbach basis of $A(X)$ with bi-orthogonal functionals $(e_i^*)_{i=1}^d$.  In particular, $\|e_i^*\otimes e_i\|=1$ for all $1\leq i\leq d$ and $\sum e^*_i\otimes e_i$ is the identity on $A(X)$. For each $1\leq j\leq dm$, we let $x_j=e_r$ and $f_j=m^{-1}A^*e_r^*$ where $j=qd+r$ for some $0\leqslant q< d$ and $1\leqslant r\leqslant m$.  Let $x\in X$ and $0\leqslant q\leqslant m$.  Then,
$$
\sum_{j=1}^{qd} f_j\otimes x_j=q\sum_{r=1}^d \frac{1}{m}(A^*e^*_r)\otimes e_r=\frac{q}{m}\big(\sum_{r=1}^d e^*_r\otimes e_r\big) A=\frac{q}{m}A.
$$

Thus the first part of \eqref{E:Pel} holds.  To prove the second part we let $0\leqslant q<m$, and $1\leqslant r\leqslant m$.  We have that,
$$
\Big\|\sum_{j=qm+1}^{qm+r} f_j\otimes x_j\Big\|
= \Big\|\sum_{j=1}^{r}  \frac{1}{m}(A^*e^*_j)\otimes e_j\Big\|
\leqslant  \sum_{j=1}^r \frac{1}{m} \|e^*_j\otimes e_j\|\|A\|=\frac{r}{m}\|A\|
$$
Thus we have proven \eqref{E:Pel}.  We now assume that  $X^*$ has the BAP and hence by Theorem \ref{equivalent} we have that $X$ has the shrinking-BAP. Let $(A_k)$ be a sequence of finite rank operators so that $x= \sum_{k=1}^{\infty} A_kx$ for all $x \in X$ and that for each $f \in X$ we have that $\lim _n\|f\circ(I-\sum_{k=1}^nA_k)\|=0$ where $I$ is the identity operator on $X$.  Let $d_k$ be the dimension of $A_k(X)$ and let $(m_k)_{k=1}^\infty$ be a sequence of natural numbers with $d_k/m_k \to 0$.

 As described above, we construct for all $k\in\N$ a finite sequence  $(x^k_i,f^k_i)_{i=1}^{d_k m_k}$ which satisfies \eqref{E:Pel} for the finite rank operator $A_k$.   We claim that the infinite sequence  $(x_i^k, f_i^k)_{k\in\N,1\leq i\leq d_km_k}$ is a shrinking Schauder frame of $X$ when enumerated in the natural way.  We will prove that  $(x_i^k, f_i^k)_{k\in\N,1\leq i\leq d_km_k}$ is shrinking, and we note that the proof that $(x_i^k, f_i^k)_{k\in\N,1\leq i\leq d_km_k}$ is a Schauder frame follows the same argument.  Let $f\in X^*$ and $\varepsilon >0$. Choose $K\in\N$ large enough so that for all $k \geqslant K$ we have $\frac{d_k}{m_k}\|A_k\| <\varepsilon$ and $\|f-\sum_{i=1}^{k} A_i^*f\|<\varepsilon$. Thus we also have that $\|A^*_k f\|<2\varepsilon$. 
  Let $l\geq \sum_{j=1}^{K} d_j m_j$.  Then $l=\sum_{j=1}^{k-1} d_j m_j+qd_k+r$ for some $0\leqslant q < m_k$, $1\leqslant r\leqslant d_k$, and  $k> K$. 
 The approximation of $f$ by the partial sum consisting of the first $l$ terms of $\sum f(x_j^i)f_j^i$ satisfies
\begin{align*}
 \Big\|f-\Big(\sum_{i=1}^{k-1}&\sum_{j=1}^{d_i m_i}f(x_j^i)f_j^i +\sum_{j=1}^{q d_k} f(x_j^{k})f_j^{k} +\sum_{j=1}^{r} f(x_j^{k})f_j^{k}\Big)\Big\|\\
 &=\Big\|\Big(f- \sum_{i=1}^{k-1}A^*_i f\Big)-\frac{q}{m_k} A^*_k f -\sum_{j=1}^{r} f_j^{k}(f)x_j^{k} \Big\|\hspace{.5cm}\textrm{ by \eqref{E:Pel},}\\
 &\leq \Big\|f-\sum_{i=1}^{k-1} A^*_i f\Big\|+\frac{q}{m_{k}}\|A_k^* f\|+ \frac{r}{m_{k}}\|A^*_k\| \|f\|<4\varepsilon.
 \end{align*}
 
 Thus we have that $(f_i^k, x_i^k)_{k\in\N,1\leq i\leq d_km_k}$ is a Schauder frame of $X^*$ and hence $(x_i^k, f_i^k)_{k\in\N,1\leq i\leq d_km_k}$ is shrinking.
\end{proof}

\bibliographystyle{abbrv}
\bibliography{bib_source.bib}

\end{document}